\documentclass[12pt]{amsart}
\usepackage{amscd,amsthm,amssymb,amsfonts,amsmath,epsfig,epic,bbm,url, latexsym, tikz}
\usepackage[arrow,matrix,tips,2cell]{xy}
\usetikzlibrary{arrows}
\usepackage{sidecap}\sidecaptionvpos{figure}{c}
\usepackage{hyperref}
\usepackage{caption} 
\usepackage{sidecap} 
\usepackage{float} 

\theoremstyle{plain}
\newtheorem{theorem}{Theorem}
\newtheorem{lemma}[theorem]{Lemma}
\newtheorem{proposition}[theorem]{Proposition}

\theoremstyle{definition}

\theoremstyle{remark}

\newcommand{\Z}{\mathbb Z}    
\newcommand{\R}{\mathbb R}    
\newcommand{\C}{\mathbb C}    
\newcommand{\PP}{\mathbb P}   
\newcommand{\T}{\mathbb T}    

\newcommand{\TT}{\mathbb{T}}

\newcommand{\pp}{\mathcal{P}}
\newcommand{\cc}{\mathcal{C}}

\newcommand{\suchthat}{\ : \ }
\newcommand{\<}{\langle}   
\renewcommand{\>}{\rangle} 

\newcommand{\am}{{\mathcal{A}}}

\newcommand{\bsd}{\operatorname{bsd}}
\newcommand{\dsd}{\operatorname{dsd}}

\newcommand{\ignore}[1]{\relax}

\newcommand{\Conv}{\operatorname{Conv}}

\begin{document}

\title{Tailoring a pair of pants: the phase tropical version}

\author{Ilia Zharkov}
\address{Kansas State University, 138 Cardwell Hall, Manhattan, KS 66506}
\email{zharkov@ksu.edu}

\begin{abstract}
We show that the phase tropical pair-of-pants $\pp^\circ \subset {(\C^*)^n}$ is (ambient) isotopic to the complex pair-of-pants $P^\circ \subset {(\C^*)^n}$. This paper can serve as an addendum to \cite{RZ} where an isotopy between complex and ober-tropical pairs-of-pants was shown. Thus the all three versions are isotopic.
\end{abstract}

\maketitle

{\hypersetup{linkbordercolor=white} 
\footnote{The research of the author is partially supported by the Simons Collaboration grant A20-0125-001. Also the author thanks Gabe Kerr and Helge Ruddat for some valuable discussions.
}
}

\section{Introduction}
The $(n-1)$-dimensional {pair-of-pants} $P^\circ \subset (\C^*)^n$ is 
the main building block for many problems in complex and symplectic geometries. Its projection  under the $\log |z|$ map is called the \emph{amoeba} and its projection via the argument map is the \emph{coamoeba}. The phase tropical pair-of-pants $\pp^\circ  \subset (\C^*)^n$ is the fibration over the tropical hyperplane with fibers over tropical strata given by the corresponding coamoebas. It is natural to consider closed spaces so we compactify $ (\C^*)^n$ to $\Delta\times \T^n$, where $\Delta$ is the standard $n$-simplex and $\T^n$ is the $n$-torus. 
The main result of the paper is Theorem \ref{thm:main} which states that the closures $\pp$  and $P$ of the two versions of pairs-of-pants in $\Delta\times \T^n$ are (ambient) isotopic (in the PL category).

Instead of trying to build an isotopy explicitly we build regular cell decompositions of both pairs and show that they are homeomorphic. The cell structures respect the natural stratification of $\Delta \times \T^n$, so the homeomorphisms will glue well at the boundary. Thus with a tiny bit of effort the isotopy can be extended to any general affine hypersurface by using the pair-of-pants decomposition of \cite{PP} and \cite{Viro11}. 

The main application of the isotopy is to address the following question in mirror symmetry. Given an integral affine manifold with singularities we want to build a topological SYZ fibration \cite{SYZ} with discriminant in codimension~2 (rather than codimension one), see \cite{gross} for the quintic 3-fold case. To compare with the ober-tropical version, see \cite{RZ2}, the phase tropical has the advantage is that no unwiggling is required when gluing local models. A disadvantage, if any, is that the singular fibers, though still half-dimensional, are more complicated.

\section{CW-structure of the complex and phase tropical pair of pants}

\subsection{Notations} 
We set $\hat n:= \{0,\dots,n\}$.
We will think of $(\C^*)^{n} \cong (\C^*)^{n+1}/\C^*$ as the product $\Delta^\circ \times \T^n$ where $\Delta^\circ$ is the interior of the $n$-simplex 
$$\Delta:=\left\{(x_0, \dots, x_n) \in \R^{n+1} \suchthat x_i\ge 0, \ \sum x_i =1 \right\}
$$ 
and $\T^n := (\R/2\pi \Z)^{n+1}/(\R/2\pi \Z) $ is the $n$-torus with homogeneous coordinates $[\theta_0, \dots, \theta_n]$. It is more natural to work with closed spaces so we will compactly $(\C^*)^{n}$ to  $\Delta \times \T^n$ and all subspaces in $(\C^*)^{n}$ by taking their closures in $\Delta \times \T^n$. We will denote by $\bsd\Delta$ the first barycentric subdivision of $\Delta$. We will also consider the {\bf dualizing} subdivision
$\dsd \Delta$ which is a coarsening of $\bsd\Delta$ by combining all simplices in $\bsd\Delta$ from a single interval $[I,J]$ together. 
That is, a cell $\Delta_{IJ}$ in $\dsd \Delta$ is
$$\Delta_{IJ}:= \Conv\{ \hat \Delta_K \suchthat I\subseteq K \subseteq J \},
$$
where $\hat \Delta_K$ stands for the barycenter of $\Delta_K$. 

The {\bf hypersimplex} $\Delta^n(2) \subset \Delta^n$ is obtained from the ordinary simplex by cutting the corners half-way. That is, 
$$\Delta^n(2):=\{(x_0,\dots,x_n) \in \R^{n+1}  \suchthat \sum x_i=2\pi  \text{ and } 0\leq x_i \leq \pi \}.
$$
We will use $2 \pi=1$ for the amoeba and $2\pi=6.28\dots$ for the coamoeba.

\subsection{Two stratifications of $\Delta \times {\T}^n$} 
The stratification $\{\Delta_J\}$ of the simplex $\Delta$ is given by its face lattice, namely by the non-empty subsets $J\subseteq \hat n$ with $\dim \Delta_J=|J|-1$.
The most refined decomposition of $\Delta$ we will ever need is the $\dsd \Delta$ whose faces  $\Delta_{IJ}$ are the pairs $I\subseteq J\subseteq \hat n$. The face lattice is given by inclusion on the $J$'s and the reverse inclusion on the $I$'s.

On the torus side the set of the hyperplanes 
$$\theta_i=\theta_j,\ i,j\in \hat n,
$$ 
stratifies $\T^n$ by cyclic orderings of the points $\theta_0, \dots, \theta_n$ on the circle. The strata $\T_\sigma$ are labeled by {\bf cyclic partitions} $\sigma=\< I_1, \dots, I_k\>$ of the set $\hat n$, that is, $\hat n= I_1 \sqcup \dots \sqcup I_k$ and the sets $I_1, \dots, I_k$, called the {\bf parts} of $\sigma$, are cyclically ordered. The elements within each part $I_i$ are not ordered. If all parts are 1-element sets then we will call this partition {\bf maximal} and write $\sigma=\<i_0, \dots, i_n\>$. 
A cyclic partition $\sigma=\< I_1, \dots, I_k\>$ can be depicted by marking $|\sigma|:=k$ points on a circle, called {\bf vertices} of $\sigma$, and labeling the arcs between them by parts in $\sigma$ in the counter clockwise order. Any coarsening of $\sigma$ is specified by a subset of vertices of $\sigma$.

Each $\T_\sigma\subset \T^n$ can be thought of as the interior of the simplex 
\begin{equation}\label{eq:lift}
\Delta_\sigma:=\left\{(\alpha_1, \dots, \alpha_k) \in \R^{k} \suchthat \alpha_i\ge 0, \ \sum \alpha_i =2\pi \right\}
\end{equation}
Here the coordinates $\alpha_i$ play the r\^ole of differences between the consecutive in the order $\sigma$ original arguments $\theta$'s. More precisely $\T_\sigma$ embeds into $\T^n$ via $\theta_i=\theta_j$ for $i,j\in I_s$ and $\theta_i+\alpha_s=\theta_j$ for $i\in I_s ,j\in I_{s+1}$, where we assumes the periodic indexing $I_{s+k}=I_s$.

The closure $\overline \T_\sigma$ of each $\T_\sigma$ in $\T^n$ lifts to the universal cover of $\T^n$ as a simplex $\Delta_\sigma$ in the cube  $[0,2\pi]^n$. The covering  map $\Delta_\sigma \to \overline \T_\sigma$ is one-to-one away from the vertices of $\Delta_\sigma$ which are all mapped to $\{0\}$, the single vertex of $\overline \T_\sigma$. 
We can then distinguish the preimages of $\{0\}$ in  $\Delta_\sigma$ by specifying a vertex from the set of vertices of $\sigma$.

There is a (twice)$^n$ more refined {\bf alcove} decomposition $\am\, \T^n$ of $\T^n$ by the hyperplanes 
\begin{equation}\label{eq:alcove}
\theta_i - \theta_j \in \pi \Z, \text{ for all pairs } i,j \in \hat n.
\end{equation}
One can parametrize the alcoves $\am_\nu$ by non-empty {\bf nets of chords} $\nu$ as follows, see \cite{KZ} for details. Draw a circle in the plane, which we assume to be (counter-clockwise) oriented. A {\bf chord} is an interval (possibly an infinitesimal tangent) in the circle connecting two points, to be vertices of the partition $[\nu]$. We say that a non-empty collection $\nu$ of chords in the disk with at most $n+1$ vertices is a {\bf net} if any two chords intersect (possibly at a vertex). Then we label the arcs between the vertices by disjoint subsets $I_s$ such that $\bar n=I_1\sqcup \dots \sqcup I_k$, they form a cyclic partition $[\nu]=\<I_1, \dots, I_k\>$.

A net of chords $\nu$ defines an alcove $\am_\nu$ by defining the relations among the arguments $\theta_i$ as follows.
If two elements  $i,j\in \hat n$ are separated by 
\begin{enumerate}
\item no chords in $\tau$, that is $i,j$ belongs to the same $I_s$ in $\sigma(\tau)$, then $\theta_i=\theta_j$;
\item all chords in $\nu$, then $\theta_i - \theta_j = \pi \mod 2\pi$;
\item some but not all chords in $\nu$, then all non-separating chords define the same counter clockwise order (otherwise, they would not intersect), say $i$ comes before $j$, then $\theta_j -\theta_i \in [0,\pi] \mod 2\pi$. 
\end{enumerate}
%
We will denote the alcove decomposition of any simplex $\Delta_\sigma$ by $\am\, \Delta_\sigma$.

Finally we combine the stratification of $\Delta$ by $\Delta_J$ with the stratification of $\T^n$ by either $\T_\sigma$ or $\am_\nu$ to get the product stratification of $\Delta\times \T^n$. 
We illustrate cells $\Delta_J\times \T_\sigma$ and $\Delta_J\times \am_\nu$ in $\Delta\times \T^n$ by marking the arcs between vertices of $\sigma$, respectively drawing the net~$\nu$, and {\underline {underlining}} the elements in ${ J}\subseteq \bar n$, see Fig.~\ref{fig:anynet}.
\begin{figure}[h]
\centering
\begin{tikzpicture}[scale=.5]
\draw (-8,0) circle [radius=2];
\draw [fill] (-10,0) circle [radius=.1];
\draw [fill] (-9,1.732) circle [radius=.1];
\draw [fill] (-7,1.732) circle [radius=.1];
\draw [fill] (-6,0) circle [radius=.1];
\draw [fill] (-7,-1.732) circle [radius=.1];
\draw (-9,-1.732) circle [radius=.1];

\node [right] at (-6.268,-1) {$\underline 0$};
\node [below] at (-8,-2) {$\underline 5$};
\node [left] at (-9.732,-1) {$4$};
\node [left] at (-9.732,1) {$\underline 3$};
\node [above] at (-8,2) {$2$};
\node [right] at (-6.268,1) {$1$};

\draw (0,0) circle [radius=2];
\draw [fill] (2,0) circle [radius=.1];
\draw [fill] (1,1.732) circle [radius=.1];
\draw [fill] (-1,1.732) circle [radius=.1];
\draw [fill] (-2,0) circle [radius=.1];
\draw (-1,-1.732) circle [radius=.1];
\draw [fill] (1,-1.732) circle [radius=.1];

\node [left] at (-1.732,-1) {$4$};
\node [below] at (0,-2) {$\underline 5$};
\node [right] at (1.732,-1) {$ \underline 0$};
\node [right] at (1.732,1) {$1$};
\node [above] at (0,2) {$2$};
\node [left] at (-1.732,1) {$ \underline 3$};

\draw [thick, blue] (-2,0)--(2,0);
\draw [thick, blue] (-2,0)--(1,1.732);
\draw [thick, blue] (-2,0)--(-1,1.732);
\draw [thick, blue] (-1,1.732)--(1,-1.732);

\draw (8,0) circle [radius=2];
\draw [fill] (10,0) circle [radius=.1];
\draw [fill] (9,1.732) circle [radius=.1];
\draw [fill] (7,1.732) circle [radius=.1];
\draw [fill] (6,0) circle [radius=.1];
\draw (7,-1.732) circle [radius=.1];
\draw [fill] (9,-1.732) circle [radius=.1];

\node [left] at (6.268,-1) {$4$};
\node [below] at (8,-2) {$ \underline 5$};
\node [right] at (9.732,-1) {$ \underline 0$};
\node [right] at (9.732,1) {$1$};
\node [above] at (8,2) {$2$};
\node [left] at (6.268,1) {$ \underline 3$};

\draw [thick, blue] (6,0)--(10,0);
\draw [thick, blue] (6,0)--(9,1.732);
\draw [thick, blue] (6,0)--(7,1.732);
\draw [thick, blue] (6,0)--(9,-1.732);
\draw (5.2,0) [thick, blue] ellipse (.8 and .4);

\end{tikzpicture}
\caption{Example: $J=\{0,3,5\}$, $\sigma=\<1,2,3,\{45\},0\>$ and two nets of chords $\nu_{1,2}$ with $[\nu_{1,2}]=\sigma$. The bold points on the circle are vertices. Tangent chord on the right is depicted as an outside loop.} 
\label{fig:anynet}
\end{figure}
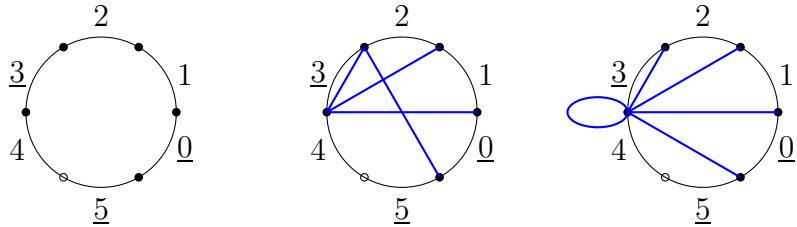

\subsection{A cell decomposition of the complex pair of pants} 
The $(n-1)$-dimensional {\bf pair-of-pants} $P^\circ$ is the complement of $n+1$ generic hyperplanes in $\C\PP^{n-1}$. 
By an appropriate choice of coordinates we can identify $P^\circ$ with the affine hypersurface in $(\C^*)^{n+1}/\C^*$ given in homogeneous coordinates by 
$$z_0+z_1+\dots + z_n=0.
$$ 
We define the {\bf compactified pair-of-pants} $P$ to be the closure of $P^\circ$ in $\Delta \times \TT^n$ via the map 
$$\mu_1\times \mu_2: (\C^*)^{n+1}/\C^* \to \Delta \times \TT^n , \quad 
[z_0,\dots ,z_n] \mapsto \left(\frac{|z_0|}{\sum |z_i |},\dots, \frac{|z_n|}{\sum |z_i |}; [\arg z_0, \dots, \arg z_n] \right).
$$
$P$ is a manifold with boundary, and it can be thought of as a real oriented blow-up of $\C\PP^{n-1}$ along its intersections with the coordinate hyperplanes in $\C\PP^n$.

The image $\mu_1(P)\subset \Delta$ is called the (compactified) {\bf amoeba} of the hypersurface $P$, see \cite{GKZ}. One can easily identify $\mu_1(P)$ with the hypersimplex $\Delta(2)$, since the only restrictions on lengths of the $z_i$ are given by the triangle inequalities. That is, if we normalize the perimeter to be 1, then $\mu_1(P)$ is cut out by the inequalities $|z_i|\le1/2$.

The image $\mu_2(P)\subset \T^n$ is called the {\bf coamoeba} of the hypersurface $P$. We denote its restriction to the stratum $\overline \T_\sigma$ by $\cc_\sigma$. In fact, since the covering map $\Delta_\sigma\rightarrow\overline{\T^n_{\sigma}}$ is bijective away from the vertices of $\Delta_\sigma$ and $\cc_\sigma$ misses $\{0\}\in \T^n$, we may view $\cc_\sigma$ as a subset of the simplex $\Delta_\sigma$. Interestingly, the coamoeba cell $\cc_\sigma\subset \Delta_\sigma$ is cut out by the inequalities $\alpha_s\leq \pi, \ \sum \alpha_s=2\pi$, that is $\cc_\sigma$ is also the  hypersimplex $\Delta_\sigma (2)\subset \Delta_\sigma$.

We will also consider {\bf partial} coamoebas $\cc^I\subset \T^n$ for $I\subset \hat n$ defined as the closure of the image $\mu_2(P^I)\subset \T^n$, where $P^I\subset (\C^*)^{n+1}/\C^*$ is a hypersurface given by $\sum_{i\in I} z_i=0$. Any partial coamoeba is the product of a lower dimensional coamoeba with a complementary dimensional torus.

The product stratification $(J, \sigma)$ of $\Delta\times \T^n$ induces the subdivision of $P$ which was shown in \cite{KZ} to be a regular CW-complex. Let us briefly describe the face lattice  $\{P_{\sigma,J}\}$ of this complex.
We say that $\sigma$ {\bf divides} $J$ if $J$ contains elements in at least two parts of $\sigma=\<I_1, \dots, I_{k}\>$. $P_{\sigma,J}$ is non-empty if and only if $\sigma$ divides $J$, then the dimension of $P_{\sigma,J}$ is $|\sigma| + |J| -4$. This, in particular, means that $|\sigma|\ge 2$ and $|J|\ge 2$.  $P_{\sigma',J'}$ is a face of $P_{\sigma,J}$ if $J'\subseteq J$ and $\sigma'$ is a coarsening of $\sigma$.

Again since the covering map $\Delta_\sigma\rightarrow\overline{\T^n_{\sigma}}$ is bijective away from $\{0\}\in \T^n$ and $P_{\sigma,J}$ does not 
have any points lying over $\{0\}$, we may view it sitting in the product of simplices
$$P_{\sigma,J} \subseteq \Delta_J \times \Delta_\sigma.$$
It was shown in \cite{RZ} that $(\Delta_J \times \Delta_\sigma, P_{\sigma,J})$ is homeomorphic to the standard ball pair for every $(J,\sigma)$. Our main goal is to prove the same result for the phase tropical case.

\subsection{Phase tropical pair-of-pants as a CW complex} 
Consider the {\bf spine} $H$ of the amoeba $\mu_1(P)$ (also known as the tropical hyperplane) which is a polyhedral subcomplex of $\dsd\Delta$ defined as 
$$H=\{(x_0, \dots ,x_n) \in \Delta \suchthat x_i=x_j\ge x_k \text{ for some } i\ne j \text{ and all } k\ne i,j\}.
$$ 
Its faces $H_{IJ}$ are cubes of dimension $|J|-|I|$ parameterized by pairs of subsets $I\subseteq J\subseteq \hat n$ with $|I|\ge 2$. Namely, $H_{IJ}$ is defined by $x_i=x_{i'} \ge x_k$ for all $i,i' \in I , k\not\in I$ and $x_j=0, j\not\in J$.

The {\bf phase tropical pair-of-pants} $\pp \subset \Delta \times \TT^n$ is the union
$$\pp: = \bigcup_{I\subseteq J} (H_{IJ} \times \cc^I).
$$
Similar to the complex pair-of-pants $P$ the stratification $\{\Delta_J\times \TT^n_\sigma\}$ of $\Delta\times \TT^n$ induces a stratification $\{\pp_{\sigma, J}\}$ of $\pp$. It was shown in \cite{KZ} that $\{\pp_{\sigma, J}\}$ form a regular CW complex isomorphic to $\{P_{\sigma, J}\}$ which proves that $P$ and $\pp$ are homeomorphic.

Again, since each $\pp_{\sigma, J}$ does not touch the vertices of $\Delta_\sigma$ we will view it as sitting inside the product of two simplices 
$$\pp_{\sigma, J}\subset \Delta_J \times \Delta_\sigma.
$$
Our main goal will be to show that this is the standard ball pair.

Let us briefly recall the polyhedral structure of $\pp_{\sigma,J}$ in this complex, see \cite{KZ} for details. 
For a subset $I\subseteq \hat n$ we say a chord in $\nu$ {\bf divides} $I$ if $I$ does not lie on one side of it. Then $\pp_{\sigma,J}$ is the union of products $H_{IK}\times \am_\nu$, such that $I\subseteq K\subseteq J$, $[\nu]$ is a coarsening of $\sigma$  and each chord in $\nu$ {\bf divides}~$I$. 
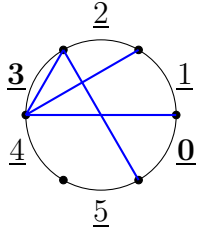
\begin{figure}[h]
\centering
\begin{tikzpicture}[scale=.5]
\draw (0,0) circle [radius=2];
\draw [fill] (2,0) circle [radius=.1];
\draw [fill] (1,1.732) circle [radius=.1];
\draw [fill] (-1,1.732) circle [radius=.1];
\draw [fill] (-2,0) circle [radius=.1];
\draw [fill] (-1,-1.732) circle [radius=.1];
\draw [fill] (1,-1.732) circle [radius=.1];

\node [left] at (-1.732,-1) {$\underline 4$};
\node [below] at (0,-2) {$\underline 5$};
\node [right] at (1.732,-1) {$\bf \underline 0$};
\node [right] at (1.732,1) {$\underline 1$};
\node [above] at (0,2) {$\underline 2$};
\node [left] at (-1.732,1) {$\bf \underline 3$};

\draw [thick, blue] (-2,0)--(2,0);
\draw [thick, blue] (-2,0)--(1,1.732);
\draw [thick, blue] (-2,0)--(-1,1.732);
\draw [thick, blue] (-1,1.732)--(1,-1.732);
\end{tikzpicture}
\caption{A face $H_{IJ}\times \am_\nu$ in $\pp_{\sigma,\hat n}$ for $\sigma=\{0,1,2,3,45\}$ and ${I}=\{0,3\}$ in {\bf bold}.} 
\label{fig:net}
\end{figure} 

The face lattice is $(I,K,\nu)\preceq (I',K',\nu')$ if $I\supseteq I'$, $K\subseteq K'$ and $\nu\subseteq \nu'$. The dimension of the $(I,K,\nu)$-stratum is given by 
\begin{equation}\label{eq:dimension}
\dim H_{IK} + \dim \am_\nu = (|K|-|I|) + (|\nu|-1).
\end{equation}

The final remark we would like to make before going into the proof section is that the alcove subdivision is not the coarsest polyhedral structure one can put on $\pp$ but it is the coarsest one which refines the $\sigma$-stratification of the $\T^n$-factor.

\section{An isotopy}\label{section3}
It was shown in \cite{KZ} that $\pp$ is a topological manifold homeomorphic to the complex pairs-of-pants. We will prove the relative (much stronger) version of this homeomorphism, which is the main result of the paper:
\begin{theorem}\label{thm:main}
The two spaces $P$ and $\pp$ are (ambient) isotopic in $\Delta\times \TT^n$. An isotopy can be chosen such that it respects the stratification $\{\Delta_J\times \TT^n_\sigma\}$.
\end{theorem}

The key ingredient for the proof of Theorem \ref{thm:main} is a compatible collection of homeomorphisms between the cell pairs $(\Delta_J\times\Delta_\sigma, P_{\sigma, J})$ and $(\Delta_J \times\Delta_\sigma, \pp_{\sigma, J})$. More precisely, we will show that both are unknotted ball pairs.

\subsection{Unknotted ball pairs} Here we collect some basic results from PL topology which we will need to prove the isotopy. A ball pair $(B^q, B^{m})$ is {\bf proper} if $\partial B^{m} =B^{m} \cap \partial B^q$.
The {\bf standard ball pair} $B^{q, m}$ is the pair of cubes $([-1,1]^q, \{0,0\}\times[-1,1]^{m})$. A ball pair is {\bf locally flat} if any point has a neighborhood homeomorphic to a neighborhood of a point in the standard pair. We say that a ball pair is {\bf unknotted} if it is homeomorphic to the standard ball pair. We will be mainly concerned with ball pairs of codimension 2. 

\begin{proposition}[See, e.g. \cite{RS}, Chapters 4 and 7] \label{prop:unknotted} 
In the PL category
\begin{enumerate}
\item
For $q\ne 4$ a locally flat ball pair $(B^{q}, B^{q-2})$ is unknotted if and only if $B^{q}\setminus B^{q-2}$ has the homotopy type of a circle.

\item A proper ball pair $(B^4,B^2)$ is unknotted if it is the cone over the locally flat sphere pair $(S^3,S^1)$ with $S^3 \setminus S^1$ homotopic to a circle.

\item 
A homeomorphism between the boundaries of unknotted balls extends to their interior. Moreover one can choose the extension to agree with any given extension on the subball.
\end{enumerate}
\end{proposition}

\subsection{Proof of the main theorem}
The main building block for the proof of Theorem~\ref{thm:main} is a homeomorphism of the pairs $(\Delta_J\times\Delta_\sigma , P_{\sigma, J})$ and $(\Delta_J\times\Delta_\sigma , \pp_{\sigma, J})$.

\begin{proposition}[\cite{RZ}, Prop. 11 and 12]\label{prop:RZ}
The ball pair $(\Delta_J\times\Delta_\sigma , P_{\sigma, J})$ is unknotted.
\end{proposition}

We will prove the same result for phase-tropical ball pair  $(\Delta_J\times\Delta_\sigma , \pp_{\sigma, J})$. First, by looking at the polyhedral structure of $\pp_{\sigma, J}$ one can easily observe that $(\Delta_J\times\Delta_\sigma , \pp_{\sigma, J})$ is a proper ball pair.

\begin{lemma}\label{lemma:flat}
The pair $(\Delta_J\times\Delta_\sigma , \pp_{\sigma, J})$ is locally flat.
\end{lemma}
\begin{proof}
Since $\pp_{\sigma, J}$ is a polyhedral subcomplex of $\dsd\Delta_J\times \am \,\Delta_\sigma$ it is enough to consider the local fan at a vertex of $\pp_{\sigma, J}$. Let $v$ be a vertex of $\cc_\sigma$ which corresponds to a 2-partition $\<I_-,I_+\>$, a coarsening of $\sigma$, and let $K$ be a subset of $J$ with non-empty $K\cap I_{\pm}$. To avoid cumbersome notations we assume $K=J$ and $\sigma$ is maximal, we let $J_{\pm}:=J\cap I_{\pm}$. Any non-maximal case is the product of a lower-dimensional maximal one with the simplicial cone on $J\setminus K$.

The local fan $\pp^v_{\sigma, J}$ of $\pp_{\sigma, J}$ at the vertex $\hat \Delta_J \times v$ projects to the $\Delta_J$ factor as a locally flat codimension 1 fan $H_{v,J}\cong \R^{|J|-2}$ isomorphic to the product of the normal fans to simplices $\Delta_{J_\pm}$.
Thus the problem is reduced to a codimension 1 ball pair $(H_{v,J}\times\Delta_\sigma, \pp^v_{\sigma, J})$ which is related to the Sch\"onflies conjecture. We will avoid the inductive dependence on dimension 4 and show the local flatness of $(H_{v,J}\times\Delta_\sigma, \pp^v_{\sigma, J})$ explicitly.

We will draw the circle with arcs labelled by the parts in $\sigma$ such that $J_-$ is on the bottom and $J_+$ is on top from the horizontal chord $v$ and order the parts in $\sigma$ from right to left, see Fig. \ref{fig:positive}. The codimension 1 subball $\pp^v_{\sigma, J}$ breaks the ambient ball $H_{v,J}\times\Delta_\sigma$ into two connected components, positive and negative. A cell $H_I \times \am_\nu$ does not belong to $\pp^v_{\sigma, J}$ if there are chords in $\nu$ not dividing $I$. Since all chords intersect $I$ lies on same side for all non-dividing chords in the above right-to-left order. If $I$ is on the left then we call $H_I \times \am_\nu$ positive, and otherwise we call it negative. E.g., in Fig. \ref{fig:positive} with $\nu$ including the 4 chords the cell $H_{03}\times \am_\nu$ is positive and the cell $H_{02}\times \am_\nu$ is negative (only one chord is non-dividing in both cases).

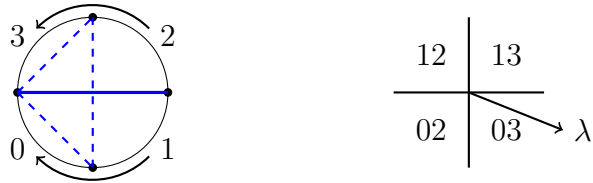
\begin{figure}[h]
\centering
\begin{tikzpicture}[scale=.5]
\draw (0,0) circle [radius=2];
\draw [fill] (2,0) circle [radius=.1];
\draw [fill] (-2,0) circle [radius=.1];
\draw [fill] (0,2) circle [radius=.1];
\draw [fill] (0,-2) circle [radius=.1];

\node [left] at (-1.5,-1.5) {$0$};
\node [right] at (1.5,-1.5) {$1$};
\node [right] at (1.5,1.5) {$2$};
\node [left] at (-1.5,1.5) {$3$};

\draw [->, thick ] (1.5,-1.7) to [out=225 ,in=-45] (-1.5,-1.7);
\draw [->, thick ] (1.5,1.7) to [out=-225 ,in=45] (-1.5,1.7);

\draw [very thick, blue] (-2,0)--(2,0);
\draw [dashed, thick, blue] (-2,0)--(0,2);
\draw [dashed, thick, blue] (-2,0)--(0,-2);
\draw [dashed, thick, blue] (0, -2)--(0,2);

\draw [thick] (8,0)--(12,0);
\draw [thick] (10,-2)--(10,2);
\node[right] at (12.5, -1) {$\lambda$};
\draw [->, thick] (10,0)--(12.5,-1);
\node at (11, -1) {$03$};
\node at (11, 1) {$13$};
\node at (9, 1) {$12$};
\node at (9, -1) {$02$};

\end{tikzpicture}
\caption{$I_-\{0,1\}$, $I_+\{2,3\}$ and a vector $\lambda$ in $\R^{|J|-2}$.} 
\label{fig:positive}
\end{figure}

Let $w$ be a vector parallel to the edge of the simplex $\Delta_\sigma$ containing $v$. Here we will have to make a choice of the direction of $w$, which is equivalent to choosing which part is positive in the 2-partition $\<I_-,I_+\>$. Then we choose a vector $\lambda$ in the product of the normal fans to simplices $\Delta_{J_\pm}$ which as a linear functional defines our right-to-left order on vertices of each simplex $\Delta_{J_\pm}$. Then we claim that $w+\lambda$ defines the desired product structure on $(H_{v,J}\times\Delta_\sigma, \pp^v_{\sigma, J})$. Geometrically it means that $w+\lambda$ ``pokes'' through $\pp^v_{\sigma, J}$ from the negative side to the positive side of $H_{v,J}\times\Delta_\sigma$, see Fig. \ref{fig:flat}.

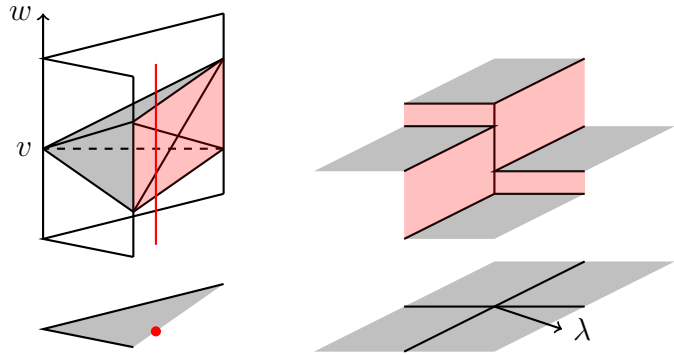
\begin{figure}[h]
\centering
\begin{tikzpicture}[scale=.6]

\draw [->, thick] (0,-2)--(0,0)--(0,3);
\draw [thick] (4,3)--(0,2)--(2,1.6);
\draw [thick] (4,-1)--(0,-2)--(2,-2.4);
\draw [thick] (4,3)--(4,-1);
\draw [thick] (2,1.6)--(2,-2.4);
\draw [thick] (4,2)--(0,0)--(2,.6)--(4,2);
\draw [thick, dashed] (4,0)--(0,0);
\draw [thick] (0,0)--(2,-1.4)--(4,0);
\draw [thick] (4,0)--(2,.56);
\draw [thick] (4,2)--(2,-1.4);
\draw [thick,red] (2.5,1.88)--(2.5,-2.13);

\fill [red, nearly transparent] (2,-1.4)--(4,0)--(4,2)--(2,.6)--(2,-1.4);
\fill [nearly transparent] (4,2)--(0,0)--(2,.6)--(4,2);
\fill [nearly transparent] (2,.6)--(2,-1.4)--(0,0);
\node [left] at (0,0) {$v$};
\node [left] at (0,3) {$w$};

\draw [thick] (4,-3)--(0,-4)--(2,-4.4);
\fill [nearly transparent] (4,-3)--(0,-4)--(2,-4.4);
\draw [fill, red] (2.5,-4.05 ) circle [radius=.1];

\draw [thick] (10,-1)--(10,1);
\draw [thick] (8,.5)--(10,.5)--(8,-.5);
\draw [thick] (12,-.5)--(10,-.5)--(12,.5);
\draw [thick] (8,1)--(10,1)--(12,2);
\draw [thick] (8,-2)--(10,-1)--(12,-1);
\fill [nearly transparent] (8,-2)--(10,-1)--(12,-1)--(10,-2);
\fill [nearly transparent] (8,.5)--(10,.5)--(8,-.5)--(6,-.5);
\fill [nearly transparent] (12,-.5)--(10,-.5)--(12,.5)--(14,.5);
\fill [nearly transparent] (8,1)--(10,1)--(12,2)--(10,2);
\fill [red, nearly transparent] (8,-2)--(10,-1)--(10,.5)--(8,-.5);
\fill [red, nearly transparent] (8,1)--(10,1)--(10,.5)--(8,.5);
\fill [red, nearly transparent] (12,-1)--(10,-1)--(10,-.5)--(12,-.5);
\fill [red, nearly transparent] (12,2)--(10,1)--(10,-.5)--(12,.5);

\draw [thick] (8,-3.5)--(12,-3.5);
\draw [thick] (8,-4.5)--(12,-2.5); 
\node[right] at (11.5, -4) {$\lambda$};
\draw [->, thick] (10,-3.5)--(11.5,-4);
\fill [nearly transparent] (6,-4.5)--(10,-2.5)--(14,-2.5)--(10,-4.5)--(6,-4.5);

\end{tikzpicture}
\caption{Neighborhood of the vertex $v$ in $\cc_\sigma$ and the fiber of $\pp^v_{\sigma, J}$ over the red line parallel to the edge and their projections along the $w$-vector.} 
\label{fig:flat}
\end{figure} 

Indeed, we just need to look at the facets of $\pp^v_{\sigma, J}$, which are of two types. Type 1: $|I|=2$ and $\nu$ is 1 chord short of being maximal. Type 2: $|I|=3$ and $\nu$ is maximal. We depict the two types and the poking by $w+\lambda$ from negative to positive cobounding facets of $H_{v,J}\times\Delta_\sigma$ in Fig. \ref{fig:type1} and \ref{fig:type2}.

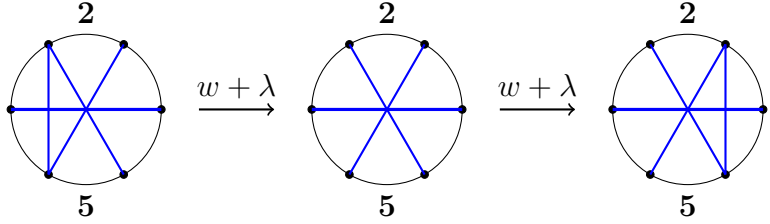
\begin{figure}[h]
\centering
\begin{tikzpicture}[scale=.5]
\draw (-8,0) circle [radius=2];
\draw [fill] (-10,0) circle [radius=.1];
\draw [fill] (-9,1.732) circle [radius=.1];
\draw [fill] (-7,1.732) circle [radius=.1];
\draw [fill] (-6,0) circle [radius=.1];
\draw [fill] (-7,-1.732) circle [radius=.1];
\draw [fill] (-9,-1.732) circle [radius=.1];

\node [below] at (-8,-2) {$ \bf 5$};
\node [above] at (-8,2) {$\bf 2$};

\draw [very thick, blue] (-10,0)--(-6,0);
\draw [thick, blue] (-9,1.732)--(-7,-1.732);
\draw [thick, blue] (-9,-1.732)--(-7,1.732);
\draw [thick, blue] (-9,-1.732)--(-9,1.732);

\draw [->, thick] (-5,0)--(-3,0);
\node [above] at (-4,0) {$w+\lambda$};

\draw (0,0) circle [radius=2];
\draw [fill] (2,0) circle [radius=.1];
\draw [fill] (1,1.732) circle [radius=.1];
\draw [fill] (-1,1.732) circle [radius=.1];
\draw [fill] (-2,0) circle [radius=.1];
\draw [fill] (-1,-1.732) circle [radius=.1];
\draw [fill] (1,-1.732) circle [radius=.1];

\node [below] at (0,-2) {$ \bf 5$};
\node [above] at (0,2) {$\bf 2$};

\draw [very thick, blue] (-2,0)--(2,0);
\draw [thick, blue] (-1,1.732)--(1,-1.732);
\draw [thick, blue] (1,1.732)--(-1,-1.732);

\draw [->, thick] (3,0)--(5,0);
\node [above] at (4,0) {$w+\lambda$};

\draw (8,0) circle [radius=2];
\draw [fill] (10,0) circle [radius=.1];
\draw [fill] (9,1.732) circle [radius=.1];
\draw [fill] (7,1.732) circle [radius=.1];
\draw [fill] (6,0) circle [radius=.1];
\draw [fill] (7,-1.732) circle [radius=.1];
\draw [fill] (9,-1.732) circle [radius=.1];

\node [below] at (8,-2) {$ \bf 5$};
\node [above] at (8,2) {$\bf 2$};

\draw [very thick, blue] (6,0)--(10,0);
\draw [thick, blue] (7,-1.732)--(9,1.732);
\draw [thick, blue] (9,-1.732)--(7,1.732);
\draw [thick, blue] (9,-1.732)--(9,1.732);

\end{tikzpicture}
\caption{Type 1: $\lambda$ is parallel to the facet, $w$ is poking through.} 
\label{fig:type1}
\end{figure} 

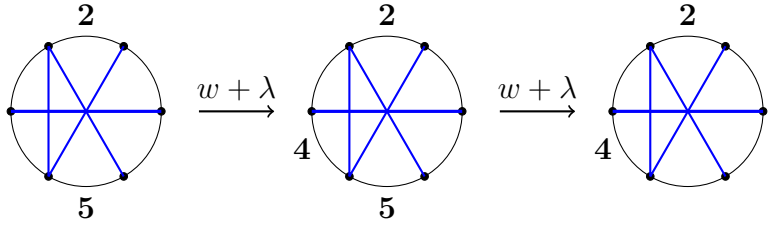
\begin{figure}[h]
\centering
\begin{tikzpicture}[scale=.5]
\draw (-8,0) circle [radius=2];
\draw [fill] (-10,0) circle [radius=.1];
\draw [fill] (-9,1.732) circle [radius=.1];
\draw [fill] (-7,1.732) circle [radius=.1];
\draw [fill] (-6,0) circle [radius=.1];
\draw [fill] (-7,-1.732) circle [radius=.1];
\draw [fill] (-9,-1.732) circle [radius=.1];

\node [below] at (-8,-2) {$ \bf 5$};
\node [above] at (-8,2) {$\bf 2$};

\draw [very thick, blue] (-10,0)--(-6,0);
\draw [thick, blue] (-9,1.732)--(-7,-1.732);
\draw [thick, blue] (-9,-1.732)--(-7,1.732);
\draw [thick, blue] (-9,-1.732)--(-9,1.732);

\draw [->, thick] (-5,0)--(-3,0);
\node [above] at (-4,0) {$w+\lambda$};

\draw (0,0) circle [radius=2];
\draw [fill] (2,0) circle [radius=.1];
\draw [fill] (1,1.732) circle [radius=.1];
\draw [fill] (-1,1.732) circle [radius=.1];
\draw [fill] (-2,0) circle [radius=.1];
\draw [fill] (-1,-1.732) circle [radius=.1];
\draw [fill] (1,-1.732) circle [radius=.1];

\node [below] at (0,-2) {$ \bf 5$};
\node [above] at (0,2) {$\bf 2$};
\node [left] at (-1.732,-1) {$\bf 4$};

\draw [very thick, blue] (-2,0)--(2,0);
\draw [thick, blue] (-1,1.732)--(1,-1.732);
\draw [thick, blue] (1,1.732)--(-1,-1.732);
\draw [thick, blue] (-1,-1.732)--(-1,1.732);

\draw [->, thick] (3,0)--(5,0);
\node [above] at (4,0) {$w+\lambda$};

\draw (8,0) circle [radius=2];
\draw [fill] (10,0) circle [radius=.1];
\draw [fill] (9,1.732) circle [radius=.1];
\draw [fill] (7,1.732) circle [radius=.1];
\draw [fill] (6,0) circle [radius=.1];
\draw [fill] (7,-1.732) circle [radius=.1];
\draw [fill] (9,-1.732) circle [radius=.1];

\node [above] at (8,2) {$\bf 2$};
\node [left] at (6.268,-1) {$\bf 4$};

\draw [very thick, blue] (6,0)--(10,0);
\draw [thick, blue] (7,-1.732)--(9,1.732);
\draw [thick, blue] (9,-1.732)--(7,1.732);
\draw [thick, blue] (7,-1.732)--(7,1.732);

\end{tikzpicture}
\caption{Type 2: $w$ is parallel to the facet, $\lambda$ is poking through.} 
\label{fig:type2}
\end{figure} 

A type 1 facet is given by the hyperplane $\theta_j-\theta_i=\pi$, where $\{i,j\}=I$. The vector $\lambda$ is parallel to it and the vector $w$ agrees with its negative/positive coorientation. A type 2 facet is given by the hyperplane $x_i=x_j$, where $I=\{i,j,k\}$ and $i,j$ are in the same part $I_-$ or $I_+$. Here the vector $w$ is parallel and $\lambda$ agrees with the coorientation. 
The right/left and negative/positive choices are made such that the projection from $H_{v,J}\times\Delta_\sigma$ along $w+\lambda$ defines a PL homeomorphism from $\pp^v_{\sigma, J}$ to its quotient image $\R^{|J|-2} \times (\cc^v_\sigma/w)$, thus giving a product structure  $H_{v,J}\times\Delta_\sigma \cong \pp^v_{\sigma, J} \times \R \<w+\lambda\>$.
\end{proof}

\begin{lemma}\label{lemma:complement}
The complement $\Delta_J\times\Delta_\sigma \setminus \pp_{\sigma, J}$ is homotopic to a circle.
\end{lemma}

\begin{proof}
Let $L_{\sigma,J}$ be the subcomplex of $\Delta_J \times \Delta_\sigma$ consisting of pairs $( \sigma', K)\preceq (\sigma,J)$ such that $\sigma'$ does {\em not} divide $K$. That is  $L_{\sigma,J}$ consists of {\em non-interlacing} pairs $(K,V)$ (that is the entire $K$ lies between two vertices), where $K\subseteq J$ and $V$ is a subset of the vertices of~$\sigma$. Lemma 9 in \cite{RZ} asserts that $L_{\sigma,J}$ collapses to a circle.
Next we will show that the complement of the open star neighborhood of $\pp_{\sigma, J}$ in $\dsd \Delta_J\times \am \, \Delta_\sigma $ is a refinement of~$L_{\sigma,J}$.

Indeed, a face $H_{IK}\times \am_\nu$ belongs to the open star neighborhood of $\pp_{\sigma, J}$ if and only it has a vertex in $\pp_{\sigma, J}$. A vertex of $H_{IK}\times \am_\nu$ is given by a subset $I'$ between $I$ and $K$ and a single chord in $\nu$. Decreasing $K$ does not do anything in terms of changing divisibility of $I$ by $\nu$. Increasing $I$, on the other hand, does. Thus $H_{IK}\times \am_\nu$ has no vertex in $\pp_{\sigma, J}$ if and and only if no chord in $\nu$ divides $K$. Since all chords are intersecting this can happen only if $K$ belong to a single part of $[\nu]$. Conversely, if $K$ is in a single part of $[\nu]$, then clearly no chord divides it.

Now notice that the union of all faces $H_{IK}$ for a fixed $K$ is just the dualizing subdivision of the face $\Delta_K\subseteq \Delta_J$. Specifying a net $\nu$ with a prefixed set of vertices $V$ is the alcove decomposition of the corresponding face $\Delta_V$ of $\Delta_\sigma$.

Finally we will show that $\Delta_J\times\Delta_\sigma \setminus \pp_{\sigma, J}$ is homotopic to the complement of the open star neighborhood of $\pp_{\sigma, J}$ in $\dsd \Delta_J\times \am \, \Delta_\sigma$. Consider a face $F:= H_{IK}\times \am_\nu$ and let $M:=F\cap \pp$, which we can assume to be a proper subcomplex of $F$. We will show that $M$ is collapsible. Then its regular neighborhood $N(M)$ in $F$ is a ball (see, e.g. \cite[Theorem 3.26]{RS}) and then its (closed) complement collapses to $\partial F \setminus N(M)$.
Thus inductively on dimension we can collapse all faces in $\Delta_J\times\Delta_\sigma \setminus \pp_{\sigma, J}$ which belong to the open star neighborhood of $\pp_{\sigma, J}$.

To see that $M$ is collapsible we note that which alcove face $\am_{\nu'}\subseteq \am_\nu$ sits over a cubical face $H_{I'K'}\subseteq H_{IK}$ depends only on~$I'$. Thus the lattice of faces with a fixed $\am_{\nu'}$ is Boolean on the $K'$-index and thus inductively collapses to just the vertex $H_{KK}$. At last we collapse the remaining alcove $\am_{\nu_0}$ which sits over the vertex $H_{KK}$, where $\nu_0\subseteq \nu$ contains all chords from $\nu$ dividing $K$.
\end{proof}

\begin{proposition}\label{prop:pt}
The ball pair $(\Delta_J\times\Delta_\sigma , \pp_{\sigma, J})$ is unknotted. 
\end{proposition}
\begin{proof}
The $\dim \Delta_J \times \Delta_\sigma \ne 4$ case follows from Lemmas \ref{lemma:flat} and \ref{lemma:complement}, and Proposition~\ref{prop:unknotted} part (1). The cases $\Delta_J^1 \times \Delta_\sigma^3$ and $\Delta_J^3\times \Delta_\sigma^1$ are trivial. Thus it only remains to show the case $\Delta_J ^2\times \Delta_\sigma^2$.
It may be possible to extend the explicit isotopy from \cite{Ca17} to an ambient one in this low-dimensional case. However it is a lot easier to compare $\pp_{\sigma, J}$ to its ober-tropical analog $\mathcal O\pp_{\sigma, J}$, which we will now describe. 

We consider the trivalent skeleton $S \subset \dsd \Delta_\sigma$ of the coamoeba triangle similar to the spine $H$ of the amoeba. Then we define $\mathcal O\pp_{\sigma, J}$ to be the union of 6 squares, see Fig. \ref{fig:ober}:
$$\mathcal O\pp_{\sigma, J} = \bigcup_{ij\ne kl} H_{ij}\times S_{kl} \subset \Delta_J \times \Delta_\sigma.
$$
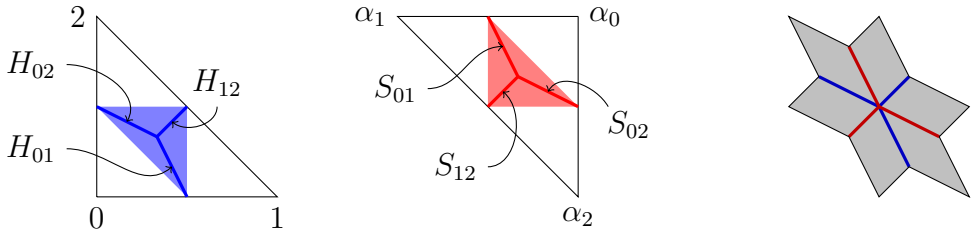
\begin{figure}[h]
\centering
\begin{tikzpicture}[scale=.8]

\draw (0,0)--(3,0)--(0,3)--(0,0);
\draw [very thick, blue] (0,1.5)--(1,1)--(1.5,0);
\draw [very thick, blue] (1.5,1.5)--(1,1);
\fill [blue, semitransparent] (0,1.5)--(1.5,0)--(1.5,1.5)--(0,1.5);

\draw (8,3)--(5,3)--(8,0)--(8,3);
\draw [very thick, red] (6.5,3)--(7,2)--(8,1.5);
\draw [very thick, red] (6.5,1.5)--(7,2);
\fill [red, semitransparent] (6.5,3)--(8,1.5)--(6.5,1.5)--(6.5,3);

\draw [very thick, blue] (12,2)--(13,1.5)--(13.5,.5);
\draw [very thick, blue] (13.5,2)--(13,1.5);
\draw [very thick, red] (12.5,1)--(13,1.5);
\draw [very thick, red] (12.5,2.5)--(13,1.5);
\draw [very thick, red] (14,1)--(13,1.5);
\draw  (14,1)--(14.5,0)--(13.5,.5)--(13,0)--(12.5,1)--(11.5,1.5)--(12,2)--(11.5,3) --(12.5,2.5)--(13,3)--(13.5,2)--(14.5,1.5)--(14,1);
\fill [nearly transparent]  (14,1)--(14.5,0)--(13.5,.5)--(13,0)--(12.5,1)--(11.5,1.5)--(12,2)--(11.5,3) --(12.5,2.5)--(13,3)--(13.5,2)--(14.5,1.5)--(14,1);

\node [below] at (0,0) {$0$};
\node [below] at (3,0) {$1$};
\node [left] at (0,3) {$2$};
\node [right] at (8,3) {$\alpha_0$};
\node [below] at (8,0) {$\alpha_2$};
\node [left] at (5,3) {$\alpha_1$};

\draw [->] (-.5,2.2) to [out=0 ,in=60] (.5,1.25);
\draw [->] (-.5,.8) to [out=0 ,in=220] (1.25,0.5);
\draw [->] (2,1.5) to [out=270 ,in=320] (1.25,1.25);
\node [left] at (-.5,2.2) {$H_{02}$};
\node [left] at (-.5,.8) {$H_{01}$};
\node [above] at (2,1.5) {$H_{12}$};

\draw [->] (5.5,1.8) to [out=0 ,in=180] (6.75,2.5);
\draw [->] (8.5,1.5) to [out=120 ,in=60] (7.5,1.75);
\draw [->] (6.5,.5) to [out=0 ,in=320] (6.75,1.75);
\node [left] at (5.5,1.8) {$S_{01}$};
\node [right] at (8.3,1.2) {$S_{02}$};
\node [left] at (6.5,.5) {$S_{12}$};

\end{tikzpicture}
\caption{The two skeleta and the ober-tropical cell $\mathcal O \pp_{\sigma,J}\subset \Delta_J ^2\times \Delta_\sigma^2 $.} 
\label{fig:ober}
\end{figure} 

The ober-tropical pair $(\Delta_J ^2\times \Delta_\sigma^2, \mathcal O \pp_{\sigma,J})$ was shown in \cite{RZ} to be unknotted, which in this dimension is almost trivial. Just look at how the boundary circle $\partial \mathcal O \pp_{\sigma,J}$  sits inside the boundary 3-sphere $\partial (\Delta_J ^2\times \Delta_\sigma^2)$, namely note that its complement collapses to a circle. Then note that the ball pair $(\Delta_J ^2\times \Delta_\sigma^2, \mathcal O \pp_{\sigma,J})$ is the cone over its boundary.

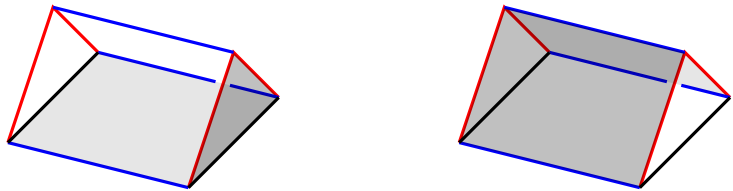
\begin{figure}[h]
\centering
\begin{tikzpicture}[scale=.6]

\draw [very thick, red] (0,0)--(1,3)--(2,2);
\draw [very thick, red] (-4,1)--(-3,4)--(-2,3);
\draw [very thick, blue] (0,0)--(-4,1);
\draw [very thick, blue] (1,3)--(-3,4);
\draw [very thick, blue] (.6,2.35)--(-2,3);
\draw [very thick, blue] (2,2)--(.92,2.27);
\draw [very thick] (0,0)--(2,2);
\draw [very thick] (-4,1)--(-2,3);

\fill [very nearly transparent]  (0,0)--(-4,1)--(-2,3)--(2,2)--(0,0);
\fill [nearly transparent]  (0,0)--(1,3)--(2,2)--(0,0);

\draw [very thick, red] (10,0)--(11,3)--(12,2);
\draw [very thick, red] (6,1)--(7,4)--(8,3);
\draw [very thick, blue] (10,0)--(6,1);
\draw [very thick, blue] (11,3)--(7,4);
\draw [very thick, blue] (10.6,2.35)--(8,3);
\draw [very thick, blue] (12,2)--(10.92,2.27);
\draw [very thick] (10,0)--(12,2);
\draw [very thick] (6,1)--(8,3);

\fill [nearly transparent]  (10,0)--(6,1)--(7,4)--(11,3)--(10,0);
\fill [very nearly transparent] (12,2)--(11,3)--(7,4)--(8,3)--(12,2);

\end{tikzpicture}
\caption{The cellular ``roofing'' move from $\pp_{\sigma,J}$ to $\mathcal O \pp_{\sigma,J}$ over the leg $H_{02}$.} 
\label{fig:isotopy}
\end{figure} 

Finally to see an isotopy from $\pp_{\sigma,J}$ to $\mathcal O \pp_{\sigma,J}$ we refine the central triangle in $\Delta_\sigma$ by its skeleton $S$ and perform three ``raising the roof''  elementary cellular moves relative boundary. The moves are in the triangular prisms over the legs of $H$. Fig. \ref{fig:isotopy} shows one of the three, the back-facing triangle is on the boundary. Any elementary move can be extended to an ambient isotopy, see e.g. \cite[Prop.~4.15]{RS}. This finishes the proof.
\end{proof}

\begin{proof}[Proof of Theorem \ref{thm:main}]
Propositions \ref{prop:RZ} and \ref{prop:pt} provide homeomorphisms of ball pairs $(\Delta_J\times\Delta_\sigma , P_{\sigma, J})$ and $(\Delta_J\times\Delta_\sigma , \pp_{\sigma, J})$ which respect the stratifications. By the Alexander trick, the homeomorphisms are homotopic to the identity on all cells. An inductive application of Proposition~\ref{prop:unknotted} part (3) on strata gives an isotopy which respects the stratification.
\end{proof}

\end{document}